\documentclass[10pt, a4paper]{article}
\usepackage[english]{babel}
\usepackage{amsmath}
\usepackage{amsthm}
\usepackage{amssymb}
\usepackage{bbm}
\usepackage{mathrsfs}
\usepackage{stmaryrd}
\usepackage[all]{xy}
\usepackage{bussproofs}

\newcommand{\arr}{\longrightarrow}
\newcommand{\iso}{\simeq}
\newcommand{\imply}{\rightarrow}

\newcommand{\E}{\mathcal{E}}

\newcommand{\PP}{\mathbb{P}}

\newcommand{\ee}{\mathbb{\varepsilon}}

\mathcode`\:="603A     
\mathchardef\colon="303A  

\mathcode`\<="4268     
\mathcode`\>="5269     
\mathchardef\gt="313E  
\mathchardef\lt="313C  

\theoremstyle{definition}
\newtheorem{deff}{Definition}[section]
\newtheorem{prop}[deff]{Proposition}

\newtheorem{lem}[deff]{Lemma}
\newtheorem{cor}[deff]{Corollary}

\date{}
\begin{document}
\title{Hilbertian toposes and $\ee$-toposes}
\author{Fabio Pasquali\\\small{Paris Diderot University}}
\maketitle
\section*{Introduction}
Hilbert's epsilon calculus is an extension of Hilbert's system for classical first order predicates logic by a term-forming operator which replaces the existential quantification. Specifically, for every well formed formula $\psi(x)$ there exists a term $\ee^x_{\psi}$, in which $x$ is not free, such that$$\exists x. \psi(x)\imply\psi [\ee^x_\psi/x]$$
Since its introduction Hilbert's Epsilon calculus has had a variety of application in many fields of mathematics, logic, informatics, linguistics, etc. One of the first approaches to the Epsilon calculus in the framework of category theory is due to J.L.Bell , who introduced in \cite{Bell} an intuitionistic typed and higher order version of it using the internal logic of elementary toposes. The calculus presented by Bell has the constraint that all epsilon terms are closed. That is to say that epsilon terms exist only for formulas with at most one free variable. This restriction is motivated by the fact that the presence of epsilon terms with one (or more) free variable implies the validity of the Axiom of Choice. Hence toposes collapse to boolean ones \cite{diaco}. The calculus was called by the author \textit{partial epsilon calculus}. In \cite{Bell} the author introduces also the notion of Hilbertian topos which is the appropriate structure to express the partial epsilon calculus. In \cite{Bell} an actual example of Hilbertian topos was not provided.\\\\
This paper can be seen as a prosecution of \cite{Bell}.  We give an example of Hilbertian topos. The relevance of this example is that it underlines the constructive nature of the partial epsilon calculus. In fact  we prove that the Martin Hyland's Effective topos \cite{Eff} is Hilbertian. Moreover we study the full epsilon calculus in toposes. It turns out that the presence of all epsilon terms is a condition stronger than the Axiom of Choice, in fact we prove that the internal logic of a topos is the full epsilon calculus if and only if the topos validate the axiom of choice and every arrow of a non-initial object to the terminal object is epic.\\\\
In section \ref{sec1} we recall the internal calculus of toposes. Then we introduce in the framework of topos logic the epsilon calculus and the partial epsilon calculus. In section \ref{sec2} we recall the Bell's definition of Hilbertian topos. Large part of the section is devoted to prove that the Effective topos is Hilbertian. The characterization theorem of those toposes whose internal logic is the full epsilon calculus is in \ref{sec3}.
\section{Epsilon calculus in toposes}\label{sec1}
We work in the internal language of a given elementary topos $\E$. 
We briefly recall it below. Details can be found in standard books of topos theory.\\\\
Suppose $\E$ is an elementary topos. We denote by $\Sigma_\E$ the many typed signature generated by the objects and the arrows of $\E$. Hence the types of $\Sigma_\E$ are the objects of $\E$ and these are closed under finite product types and power types with all $\eta/\beta$-conversions. This is to say that, up to $\eta/\beta$-conversion, the product type of $A$ and $B$ is a categorical product $A\times B$, the unite type is a terminal object $1$ of $\E$ and the power type of $A$ is a categorical power object $\PP (A)$ of $A$.\\\\
A function symbol $f$ of $\Sigma_\E$ with arity $A_1,A_2,...,A_n\arr B$ is an arrow $f:A_1\times A_2\times...\times A_n\arr B$ in $\E$. A relation symbol of $\Sigma_\E$ of arity $A_1,A_2,...,A_n$ is a monic arrow of $\E$ with codomain $A_1\times A_2\times...\times A_n$.\\\\
Contexts, i.e finite lists of typed variables $$(x_1:A_1, x_2:A_2,\dots,x_n:A_n)$$ will be denoted by capital greek letters. A term is an expression of the form $$\Gamma \mid t:A$$ where $\Gamma$ is a context. We write sequents in the following form $$\Gamma\mid \phi_1,\phi_2,...\phi_n\vdash \phi_{n+1}$$ to express that all formulas in the sequent are formulas in the context $\Gamma$.\\\\
As we work in the internal logic of an elementary topos $\E$, a sequent $$\Gamma\mid \phi_1,\phi_2,...\phi_n\vdash \phi_{n+1}$$ in the context $\Gamma$ is derivable if and only if the monic arrow with codomain $\Gamma$ that represents $\phi_1\wedge\phi_2\wedge\dots\wedge\phi_n$ factors through the monic arrow that represent $\phi_{n+1}$.
\subsubsection*{The epsilon calculus}
The internal language of an elementary topos is the (intuitionistic higher order and typed) $\ee$-calculus if it validates the following two rules
\[
\AxiomC{$\Gamma, x:A \mid \psi$}
\RightLabel{$\ee$-form}
\UnaryInfC{$\Gamma \mid \ee^x_\psi:A $}
\DisplayProof\ \ \ \ \ \ 
\AxiomC{$\Gamma, x:A \mid \psi$}
\RightLabel{$\ee$-I}
\UnaryInfC{$\Gamma \mid \exists x:A.\ \psi\vdash\psi[\ee^x_\psi/x]$}
\DisplayProof
\]
provided that $A$ is not the empty type (i.e. the initial object) and the empty type does not appear among types in $\Gamma$.\\\\
The rule $\ee$-form establish that for every monic $\psi$ with codomain $\Gamma\times A$, if $A$ is not initial, then there is an arrow $\ee_\psi:\Gamma\arr A$. This motivates the constraint on $A$ to be not the empty type. For if $A$ is the empty type $0$, i.e. $0$ is an initial object of $\E$, then using $\ee$-form we get
\[
\AxiomC{$x:0 \mid \top$}
\RightLabel{$\ee$-form}
\UnaryInfC{$*:1\mid \ee^x_\top:0 $}
\DisplayProof
\]
where $\ee^x_\top$ is an arrow $1\arr 0$. If such an arrow exists, the topos is degenerate.\\\\
If the internal language of an elementary topos validates $\ee$-form and $\ee$-I, then the axiom of choice is derivable: take a formula $a:A,b:B\mid F$, then we have a term $a:A\mid\ee^b_F: B$ such that $$a:A\mid \exists b:B.\ F(a,b) \vdash F(a, \ee^b_F(a))$$
which leads to the following
$$\top\vdash \forall a:A.\ [\exists b:B.\ F(a,b) \imply F(a, \ee^b_F(a))]$$
In toposes, if the axiom of choice is valid, the law of excluded middle is derivable \cite{diaco}. This motivates the following definition.\\\\
We say that the internal language of an elementary topos is the partial $\ee$-calculus if it validates the following two rules
\[
\AxiomC{$x:A \mid \psi$}
\RightLabel{$\ee_p$-form}
\UnaryInfC{$*:1\mid\ee^x_\psi:A $}
\DisplayProof\ \ \ \ \ \ 
\AxiomC{$x:A \mid \psi$}
\RightLabel{$\ee_p$-I}
\UnaryInfC{$\exists x:A.\ \psi\vdash\psi[\ee^x_\psi/x]$}
\DisplayProof
\]
provided that $A$ is not the empty.\\\\
In the partial $\ee$-calculus, $\ee$-terms are necessarily closed. Or, equivalently, $\ee$-terms exists only for those formulas with at most one free variable. Thus all $\ee$-terms are arrows whose domain is teh terminal object $1$. Under this constraint we can not carry out the previous argument to derive the axiom of choice. In fact in the partial $\ee$-calculus the axiom of choice is not derivable, as we shall see in next section.\\\\
We call Hilbertian toposes those toposes whose internal logic is the partial $\ee$-calculus. And we call $\ee$-toposes those toposes whose internal logic is the full $\ee$-calculus.
The former class was already studied in \cite{Bell} and we dedicate the next section to produce a relevant member of it. The latter is introduced in section \ref{sec3}. 
\section{Hilbertian toposes}\label{sec2}
Hilbertian toposes were introduced by Bell in \cite{Bell} as the appropriate structure to validate rules of partial epsilon calculus.\\\\
We shall denote by $\cdot\rightarrowtail\cdot$ and  $\cdot\twoheadrightarrow\cdot$ arrows which are monic and epic respectively.\\\\
An elementary topos $\E$ is Hilbertian if every diagram of the form
\[
\xymatrix{
X\ \ar@{>->}[r]^-{p}&A\ar@{->>}[r]&1
}
\]
can be expanded to a commutative diagram of the form
\[
\xymatrix{
X\ar@{->>}[r]&U\ \ar@{>->}[r]\ar[d]&1\ar[d]^-{\ee_p}&\\
&X\ \ar@{>->}[r]_-{p}&A\ar@{->>}[r]&1
}
\]
It is clear that if $p:X\rightarrowtail A$ is a formula in the context $A$, say $a:A\mid p(a)$ then, by the general theory of topos logic \cite{Elephant}, the arrow $U\rightarrowtail 1$ corresponds to the sentence $\exists a:A\ p(a)$. By the universal property of pullback there is an arrow from $U$ to any pullback of $p$ along $\ee_p$, which is to say that the sequent $\exists a:A\ p(a) \vdash p(\ee_p)$ is valid. Since the other direction is always derivable, the condition above is equivalent to say that $U$ is isomorphic to the pullback of $p$ along $\ee_p$. This is also the proof of the following theorem.
\begin{prop} The internal logic of an Hilbertian topos is the partial epsilon calculus exactly when every object $A$ is either initial or the arrow $A\arr1$ is epic.
\end{prop}
As remarked in \cite{Bell}, as soon as we admit $\ee$-terms with even one free variable, the axiom of choice is derivable. This is not the case for Hilbertian toposes. To underline the constructive nature of Hilbertian toposes we show that Martin Hyland's Effective topos is Hilbertian.\\\\We recall below the definition of the effective topos. This topos can be seen as the result of the tripos to topos construction applied to the realizability tripos \cite{tripinret}. As triposes do not play any role here, we give a direct presentation of the effective topos and we address the reader to \cite{Eff} for details.
\subsubsection*{The effective topos}
If $n$ and $p$ are natural numbers, we write $n.p$ to denote the application to $p$ of the partial recursive function coded by $n$, and we call $n.p$ the Kleene application of $n$ to $p$.\\\\Suppose $X$ is a set and $\phi, \psi: X\arr \PP(\mathbb{N})$ are function from $X$ to the powerset of the set of natural numbers. We say that $\phi\le \psi$ if there exists a number $n$ such that for every $x\in X$ and every $a\in \phi(x)$ the Kleene application $n.a$ is defined and $n.a\in \psi(x)$. The relation $\le$ is a preorder on $\PP(\mathbb{N})^X$ and we call it Kleene realizability order. Moreover we say that $n$ is a track of $\phi\le\psi$. The poset reflection of $\PP(\mathbb{N})^X$ is a Heyting algebra \cite{tripinret}. Binary meets are denoted by $\wedge$ and a top element by $\top_X$. Pseudo relative complements by $\imply$. The top element $\top_X$ is represented by the function on $X$ to $\PP(\mathbb{N})$ which whose values are  constantly equal to $\mathbb{N}$.\\\\
Denote by $1$ a chosen singleton set.
Members of $\PP(\mathbb{N})^1\iso \PP(\mathbb{N})$ are called sentences. We say that a sentence $\phi$ is valid when $\top_1\le\phi$.\\\\
For every set $X$, and every $\rho\in \PP(\mathbb{N})^{X\times X}$ we say that $\rho$ is a partial equivalence relation over $X$ if both the sentences $$\bigcap_{x\in X}\bigcap_{y\in X} \rho(x,y)\imply \rho (y,x)$$
$$\bigcap_{x\in X}\bigcap_{y\in X} \bigcap_{z\in X} \rho(x,y) \wedge \rho(y,z) \imply \rho (x,z)$$are valid. If we read the intersection as a universal quantification, the two sentences above express the fact that $\rho$ is symmetric and transitive.\\\\
Suppose $A$ and $B$ are sets and $\rho$ and $\sigma$ are partial equivalence relations over $A$ and $B$ respectively. $F\in\PP(\mathbb{N})^{A\times B}$ is said to be a functional relation from the pair $(A,\rho)$ to the pair $(B,\sigma)$ if all the following sentences are valid
$$\bigcap_{a\in A}\bigcap_{b\in B}F(a,b)\imply \rho(a,a)\wedge \sigma(b,b)$$
$$\bigcap_{a\in A}\bigcap_{a'\in A}\bigcap_{b\in B}\bigcap_{b'\in B}F(a,b)\wedge\rho(a,a')\wedge \sigma(b,b')\imply F(a',b')$$
$$\bigcap_{a\in A}\bigcap_{b\in B}\bigcap_{b'\in B} F(a,b)\wedge F(a,b')\imply \sigma(b,b')$$
$$\bigcap_{a\in A}[\rho(a,a)\imply \bigcup_{b\in B} F(a,b)]$$
The first and second sentences express the fact that $F$ "believes" that $\rho$ and $\sigma$ are equality predicates over $A$ and $B$ respectively. Then in particular in the first sentence $\rho(a,a)\wedge \sigma(b,b)$ plays the role of a top element. The third and fourth sentences express (if we read the union as an existential quantification) the fact that $F$ behaves like a function from $A$ to $B$, i.e. it is singled-valued and total.\\\\
The effective topos $\E ff$ is the category whose objects are pairs $(A,\rho)$ where $A$ is a set and $\rho$ is a partial equivalence relation over $A$ and an arrow $[F]:(A,\rho)\arr (B,\sigma)$ is an equivalence class of functional relations where $G \in [F]$ if and only if $G= F$ as members of the poset reflection of $\PP(\mathbb{N})^{A\times B}$.\\\\
For the rest of the section we freely confuse a morphism of $\E ff$ with any of its representatives.\\\\
Suppose $(A,\rho)$ is an object of $\E ff$. A member $P$ of $\PP(\mathbb{N})^{A}$ is said to be a strict and relational proposition over $(A,\rho)$ if the following sentences are valid
$$\bigcap_{x\in A}P(x)\imply\rho(x,x)$$
$$\bigcap_{x\in A}\bigcap_{y\in A}\rho(x,y)\wedge P(x)\imply P(y)$$
As we wrote before for $F$, these two sentences express the fact that $P$ "believes" that $\rho$ is an equality predicate. We say that a monomorphism $M$ of $\E ff$, with codomain $(A,\rho)$, is represented by $P$ if $M$ is isomorphic to a monic $\rho_P:(A,\rho_P)\arr(A,\rho)$ where $$\rho_P(x,y) = \rho(x,y)\wedge P(x)$$
The following properties holds in $\E ff$, proofs are in \cite{Eff, tripinret}.
\begin{itemize}
\item[1)] Every monomorphism $M:(A,\rho)\arr(B,\sigma)$ is isomorphic to one of the form $\sigma_P:(B,\sigma_P)\arr(B,\sigma)$ for some strict and relational proposition $P$ over $(B,\sigma)$.
\item[2)] A terminal object of $\E ff$ is $(1, \top_{1\times 1})$. And the unique arrow $1_\rho:(A,\rho)\arr (1,\top_{1\times 1})$ is the functional relation in $\PP(\mathbb{N})^{A\times 1}\iso \PP(\mathbb{N})^{A}$ determined by the following assignment $x\mapsto \rho(x,x)$.
\item[3)] Given a diagram of the form $(A,\rho)\twoheadrightarrow  (U,\xi)\rightarrowtail(1,\top_{1\times 1})$, the monomorphism $(U,\xi)\rightarrowtail(1,\top_{1\times 1})$ is represented by $$ \bigcup_{x\in A} \rho(x,x)$$
which is a strict and relational proposition over $(1,\top_{1\times 1})$.
\item[4)] For every pair of objects of $\E ff$, say $(A,\rho)$ and $(B,\sigma)$, if $f:A\arr B$ is a function such that the sentence $$\bigcap_{x\in A}\bigcap_{y\in A} \rho(x,y)\imply \sigma(f(x),f(y))$$ is valid, then the assignment $$(x,b)\mapsto \sigma(f(x),b)\wedge \rho(x,x)$$ determines an element of $\PP(\mathbb{N})^{A\times B}$, which is a functional relation from $(A,\rho)$ to $(B,\sigma)$. We denote it by $\Gamma(f)$ and we call it the graph of $f$. Thus $\Gamma(f)$ is (the representative of) a morphism of $\E ff$ from $(A,\rho)$ to $(B,\sigma)$.
\item[5)] A pullback of a monic $\sigma_P: (B,\sigma_P)\arr (B,\sigma)$ along $F:(A,\rho)\arr (B,\sigma)$ is the monic represented by $Q$ in $\PP(\mathbb{N})^A$, where $$Q(a) = \bigcup_{b\in B}F(a,b) \wedge P(b)$$
\item[6)] An arrow $E:(A,\rho)\arr (B,\sigma)$ is epic if and only if the following sentence is valid$$\bigcap_{b\in B}[\sigma(b,b)\imply\bigcup_{a\in A}E(a,b)]$$
\end{itemize}
Previous points 1) to 5) allow to prove the following lemma.
\begin{lem} Under the notation above, if $F$ is of the form $\Gamma(f)$ for some function $f:A\arr B$, then $Q = P\circ f \wedge \rho$ in the poset reflection of $\PP(\mathbb{N})^A$.
\end{lem}
\begin{proof} Consider the following equalities in the poset reflection of $\PP(\mathbb{N})^A$
\begin{equation}\notag
\begin{split}
Q(a) &= \bigcup_{b\in B}\Gamma(f)(a,b) \wedge P(b)\\
& = \bigcup_{b\in B}\sigma(f(a),b)\wedge \rho(a,a) \wedge P(b)\\
& = \bigcup_{b\in B}\sigma(f(a),b)\wedge \rho(a,a) \wedge P(f(a))\\
&= P(f(a)) \wedge \bigcup_{b\in B}\Gamma(f)(a,b)\\
&= P(f(a)) \wedge \rho(a,a)
\end{split}
\end{equation}
where in passing from second to third line we used the fact that $P$ is relational and in the last line we use the fact that $\Gamma(f)$ is functional (and hence total).
\end{proof}
\begin{cor}\label{coro} If $(A,\rho)$ is $(1,\top_{1\times 1})$ and if $F$ is of the form $\Gamma(f)$ for some function $f:1\arr B$, then $Q = P\circ f$ in the poset reflection of $\PP(\mathbb{N})^A$.
\end{cor}
\begin{proof} Straightforward.
\end{proof}
We can now prove the main proposition of the section.
\begin{prop}
$\E ff$ is Hilbertian.
\end{prop}
\begin{proof}
Consider the diagram
\[
\xymatrix{
(A,\rho_P)\ \ar@{>->}[r]^-{\rho_P}&(A,\rho)\ar@{->>}[r]^-{!_\rho}&1
}
\]
It suffices to show that there exists an arrow $\ee_{\rho_P}:(1,\top_{1\times 1})\arr (A,\rho)$ such that the pullback of $\rho_P$ along $\ee_{\rho_P}$ is isomorphic to the monic arrow in any epi-mono factorization of $!_\rho\circ\rho_P$. After point 3) above this is represented by $$\bigcup_{a\in A}\rho_P(a,a)=\bigcup_{a\in A}\rho(a,a)\wedge P(a)=\bigcup_{a\in A}P(a)$$
Since $!_\rho$ is epic we have after point 6) that $$\top_1\imply\bigcup_{a\in A}\rho(a,a)$$is a valid sentence. Which is to say that in $\PP(\mathbb{N})^1\iso \PP(\mathbb{N})$ we have $$\mathbb{N}\le \bigcup_{a\in A}\rho(a,a)$$ In other words the union on the right is not empty, i.e. there is $\overline{a}$ with $ \rho(\overline{a},\overline{a})\not=\emptyset$.\\\\
Any such $\overline{a}\in A$ determines an obvious function $\overline{a}: 1\arr A$ with the property that $$\mathbb{N}\le \rho(\overline{a},\overline{a})$$ in fact the inequality is tracked by the code of any constant function $n\mapsto x$ where $x$ is any chosen element of $\rho(\overline{a},\overline{a})$. Therefore the graph $$\Gamma(\overline{a}):(1,\top_{1\times 1})\arr (A,\rho)$$ is an arrow of $\E ff$.\\\\
After \ref{coro} we have that the pullback of $\rho_P$ along $\Gamma(\overline{a})$ is represented by $P (\overline{a})$. Then to prove that $\E ff$ is Hilbertian it suffices to show that there is $\ee:1\arr (A,\rho)$ such that $$P(\ee)=\bigcup_{a\in A}P(a)$$
If the union is empty then $P(a)=\emptyset$ for every $a\in A$, thus any choice of $\ee$ with $ \rho(\ee,\ee)\not=\emptyset$ determines the desired $\ee$-term. We proved above that at least one such arrow exists.\\\\
If the union is not empty, then there is $q\in A$ with $P(q)\not=\emptyset$. Since $P$ is strict we have that $P(q)\le \rho(q,q)$, then $ \rho(q,q)$ is not empty and $\Gamma(q)$ is the desired epsilon term.\end{proof}
\section{$\ee$-Toposes}\label{sec3}
If, on one side, Hilbertian toposes are the appropriate structure to interpret the partial epsilon calculus, $\ee$-toposes, which we introduce below, are the appropriate structure to express the full epsilon calculus. Indeed we show a little more: $\ee$-toposes are the exact structure to express the full epsilon calculus (see proposition \ref{car}). Of course, as the axiom of choice is derivable, $\ee$-toposes validate the law of excluded middle \cite{diaco}. Recall that a topos validates the axiom of choice (AC) if every epimorphism has a section.
\begin{deff}An elementary topos $\E$ is an $\ee$-topos if it validates AC and if every arrow $A\arr 1$ is epic, provided that $A$ is not initial.
\end{deff}
There is a trivial example of $\ee$-topos which is \textbf{Sets}, the topos of sets and functions. 
\begin{prop}\label{car} Let $\E$ be an elementary topos. The following are equivalent
\begin{itemize}
\item[i)] The internal logic of $\E$ is the full epsilon calculus
\item[ii)] $\E$ is an $\ee$-topos
\end{itemize}
\end{prop}
Before proving the statement, we remark that the internal logic of Hilbertian toposes is a fragment of the full epsilon calculus, then an immediate consequence of \ref{car} is that every $\ee$-topos is Hilbertian.
\begin{proof} i$\imply$ ii) We already remarked that in the full epsilon calculus AC is derivable. Now consider a non-initial object $A$. We have a proposition $\top_{1\times A}$ which is represented by the identity arrow $id_{1\times A}$. Then there exists an epsilon term $$\ee_{\top_{1\times A}}:1\arr A$$
which is necessarily a section of the unique arrow $A\arr 1$. An arrow with a section is epic.\\\\
ii$\imply$i) Suppose $\phi :X\rightarrowtail \Gamma\times A$ is monic. Then we have a diagram
\[
\xymatrix{
X\ar[d]_-{\phi}\ar@{->>}[r]^-{q}&\text{Im}(\pi_\Gamma\phi)\ar[d]^-{m}\\
\Gamma\times A\ar[r]_-{\pi_\Gamma}&\Gamma
}
\]
where the image of $\pi_\Gamma\phi$ is the monomorphism $m:\text{Im}(\pi_\Gamma\phi)\rightarrowtail \Gamma$ which represents the formula $\Gamma\mid\exists a:A\ \phi(a)$.\\\\
By AC the epimorphism $q$ has a section $s$. Moreover every monic has a complement, then $\Gamma$ is the coproduct $$\Gamma \iso \text{Im}(\pi_\Gamma\phi) + Y$$ for some object $Y$, where $m$ is one of the two canonical injections. Denote by $i_Y$ the other canonical injection and consider the diagram
\[
\xymatrix{
X\ar[r]^-{\phi}&\Gamma\times A&1\ar[l]_-{a}\\
\text{Im}(\pi_\Gamma\phi)\ar[r]_-{m}\ar[u]^-{s}&\Gamma\ar[u]^-{[\phi s, a!_Y]}&Y\ar[l]^-{i_Y}\ar[u]_-{!_Y}
}
\]
where $a$ is a section of the unique arrow $\Gamma\times A\arr 1$, which exists by AC since we assumed that those arrows are epic.
The arrow $\pi_A[\phi s, a!_Y]$ is $\ee_\phi$.\\\\
It remains to prove that the square
\[
\xymatrix{
\text{Im}(\pi_\Gamma\phi)\ar[r]^-{s}\ar[d]_-{m}&X\ar[d]^-{\phi}\\
\Gamma\ar[r]_-{<id_\Gamma,\ee_\phi>}&\Gamma\times A
}
\]
which is commutative, is also a pullback.\\\\Consider the arrows $x:Q\arr X$ and $y:Q\arr \Gamma$, with  $<id_\Gamma,\ee_\phi>y=\phi x$. The arrow $qx: Q\arr \text{Im}(\pi_\Gamma\phi)$ is the desired arrow, in fact
$$\phi sqx =<id_\Gamma,\ee_\phi>mqx=<id_\Gamma,\ee_\phi>\pi_\Gamma\phi x=<id_\Gamma,\ee_\phi>\pi_\Gamma<id_\Gamma,\ee_\phi>y=<id_\Gamma,\ee_\phi>y=\phi x$$
then $x=sqx $ since $\phi$ is monic. Analogously, using the fact that $<id_\Gamma, \ee_\phi>$ is monic, since it is a section, we have $mqx=y$
\end{proof}
One might be led to think that the condition that arrows $A\arr 1$ are epic is redundant. The following example shows that in a topos this does not follow from AC: consider the topos $\textbf{Sets}\times \textbf{Sets}$. It validates AC, but the arrow $(1,\emptyset)\arr(1,1)$ is clearly not epic. In other words, the presence of all $\ee$-terms is sufficient to prove AC, but not necessary.
\section{Conclusions}
If, on one side, the internal logic of Hilbertian toposes is intuitionistic, on the other side the internal logic of $\ee$-toposes is necessarily classic. Toposes can be seen as Heyting categories with the power objects axiom. In Heyting categories AC does not imply LEM and therefore they might constitute an appropriate setting to interpret the full intuitionistic first order epsilon calculus. Nevertheless the study of the epsilon calculus inside a Heyting category is left to future developments.


\begin{thebibliography}{1}

\bibitem{Bell}
John~L. Bell.
\newblock Hilbert's {$\epsilon$}-operator in intuitionistic type theories.
\newblock {\em Math. Logic Quart.}, 39(3):323--337, 1993.

\bibitem{diaco}
R.~Diaconescu.
\newblock Axiom of choice and complementation.
\newblock {\em Proceedings of the American Mathematical Society}, 51(1):176 --
  178, 1975.

\bibitem{Eff}
J.~M.~E. Hyland.
\newblock The effective topos.
\newblock In {\em The {L}.{E}.{J}. {B}rouwer {C}entenary {S}ymposium
  ({N}oordwijkerhout, 1981)}, volume 110 of {\em Stud. Logic Foundations
  Math.}, pages 165--216. North-Holland, Amsterdam-New York, 1982.

\bibitem{Elephant}
P.~T. Johnstone.
\newblock {\em Sketches of an elephant - A topos theory compendium}.
\newblock Clarendon Press - Oxford, 2002.

\bibitem{tripinret}
A.~M. Pitts.
\newblock Tripos theory in retrospect.
\newblock {\em Math. Struct. in Comp. Science}, 12:265--279, 2002.

\end{thebibliography}
\end{document}